\documentclass{cmslatex}

\usepackage[backref,colorlinks,linkcolor=red,anchorcolor=green,citecolor=blue]{hyperref}
\usepackage[paperwidth=7in, paperheight=10in, margin=.875in]{geometry}
 \usepackage{amsfonts,amssymb}
\usepackage{amsmath}
\usepackage{graphicx}
\usepackage{cite}
\usepackage{enumerate}
\newtheorem{remark}{Remark}[section]
\renewcommand{\theequation}{\arabic{section}.\arabic{equation}}

   \allowdisplaybreaks
\begin{document}
 \title{Plane-wave analysis of a hyperbolic system of equations with relaxation
in $\mathbb{R}^\MakeLowercase{d}$\thanks{M.V.d.H. gratefully acknowledges support from the Simons Foundation under the MATH + X program, the National Science Foundation under grant DMS-1559587, and the corporate members of the Geo-Mathematical Group at Rice University. J.-G.L. is supported by the National Science Foundation under grant DMS-1514826 and KI-Net RNMS11-07444.}}

          %For each author, make a block with the following macros:

          \author{MAARTEN V. DE HOOP\thanks{Department of Computational and Applied Mathematics, Rice University, Houston, TX 77005, USA (mdehoop@rice.edu).}
          \and  JIAN-GUO LIU\thanks{Department of Mathematics and Department of Physics, Duke University, Durham, NC 27708, USA (jliu@math.duke.edu).}
          \and PETER A. MARKOWICH\thanks{Applied Mathematics and Computer Science Program, CEMSE Division, King Abdullah University of Science and Technology, Thuwal 23955-6900, Saudi Arabia and Faculty of Mathematics, University of Vienna, Vienna A-1090, Austria (peter.markowich@kaust.edu.sa).} 
          \and NAIL S. USSEMBAYEV\thanks{Applied Mathematics and Computer Science Program, CEMSE Division, King Abdullah University of Science and Technology 
Thuwal, 23955-6900, Saudi Arabia (nail.ussembayev@kaust.edu.sa). \break Corresponding author: Nail S. Ussembayev }}

         \pagestyle{myheadings}\thispagestyle{plain} \markboth{Plane-wave analysis of a hyperbolic system of equations with relaxation
in $\mathbb{R}^\MakeLowercase{d}$}{M. V. de Hoop et al.} \maketitle

          \begin{abstract}
               We consider a multi-dimensional scalar wave equation with
  memory corresponding to the viscoelastic material described by a
  generalized Zener model. We deduce that this relaxation system is an example of a non-strictly hyperbolic system satisfying Majda's block structure condition. Well-posedness of the associated Cauchy
  problem is established by showing that the symbol of the spatial
  derivatives is uniformly diagonalizable with real eigenvalues. A
  long-time stability result is obtained by plane-wave analysis when
  the memory term allows for dissipation of energy.
          \end{abstract}

          \section{Introduction}\label{intro}
          
The theory of viscoelasticity describes materials exhibiting a
combination of both elastic solid (deformation eventually disappears
when the load is removed) and viscous (Newtonian) fluid
characteristics. Wave propagation in viscoelastic unbounded or
semi-bounded media is a relevant idealization in some important
real-world problems arising in different fields: geophysics, applied
mechanics, material science, acoustics etc.

Viscoelastic materials are modeled by constitutive laws relating the
stress to the history of the strain and entering the equation of
motion in the form of a convolution integral in time. The resulting
integro-differential equation can be written as a system of partial
differential equations with a relaxation term and described in Fourier
space as an exponential evolution operator acting on a vector
representing the initial conditions.  The system is hyperbolic when
the matrix appearing in the evolution operator is diagonalizable with
real eigenvalues and its eigenspace is complete. If, in addition, all eigenvalues are distinct, the system is said to be strictly hyperbolic.  One of the important motivations to study strictly hyperbolic systems is that they are invulnerable to perturbations by lower-order terms. Unfortunately, many interesting examples of hyperbolic systems describing various physical phenomena are not strictly hyperbolic and it is not known in general whether such systems remain hyperbolic under perturbations by lower-order terms.

 Majda and Osher \cite{majda} proved that the strict hyperbolicity assumption used in the construction of Kreiss' symmetrizer could be replaced by a weaker assumption called the "block structure condition" which is satisfied by several non-strictly hyperbolic systems including Maxwell's equations of electrodynamics, the linearized shallow water equations and the Euler equations of gas dynamics. However, each system of interest required a separate verification of this property due to the lack of a universal criterion. This was the state of affairs until M\'{e}tivier \cite{metivier} extended Majda's work establishing the block structure condition for a class of hyperbolic systems with characteristic fields of constant multiplicity. It is common to refer to such systems simply as "constantly hyperbolic," to wit \\

\begin{definition}\label{def. 1.1}
The operator \[L=\partial_{t}+\sum_{j=1}^{d}A_{j}(x,t)\partial_{x_{j}}\] with $A_{j},B:\mathbb{R}^{d}\times(0,T)\to\mathbb{M}^{N\times N}(\mathbb{R})$ is called constantly hyperbolic if there exist an integer $m\geq1$, natural numbers $l_1,\dots, l_m$ and real valued functions $\lambda_1,\dots, \lambda_m$ analytic away from the origin such that for any $\xi\in\mathbb{S}^{d-1}$ it holds that \[\det\left(\lambda I_N+\sum_{j=1}^{d}\xi_{j}A_{j}\right)=\prod_{i=1}^{m}\left(\lambda+\lambda_{i}(\xi)\right)^{l_{i}},\ \ \ \ \ \ l_1+\hdots+l_m=N\] where all the eigenvalues $\lambda_i(\xi)$ of the symbol $A(\xi)=\sum_{j=1}^{d}\xi_{j}A_{j}$ are real, semi-simple and satisfy $\lambda_1(\xi)<\hdots<\lambda_m(\xi)$. 
\end{definition} \\

 Let us reiterate: if the eigenvalues are semi-simple instead of being simple as in the case of strict hyperbolicity, and their multiplicities remain constant as $(\xi_1,\dots, \xi_d)\in\mathbb{R}^d\setminus\{0\}$ varies, then the corresponding system is called constantly hyperbolic. The notion of constant hyperbolicity is
a slight generalization of the concept of strict hyperbolicity where the
analysis is technically simpler and had allowed more extensive studies in the past. In Sec. \ref{Sec.3} we demonstrate the hyperbolicity of our relaxation system by proving that $A(\xi)$ is diagonalizable with real eigenvalues and verify that the diagonalization is well-conditioned on $\mathbb{S}^{d-1}$. More straightforwardly, hyperbolicity can be shown by appealing to the general structure of the eigenvalues of the system since constant hyperbolicity implies hyperbolicity (see Remark \ref{remark 3.1}).

In a bounded domain, existence and uniqueness of solutions can be
established using the treatment of Lions and Magenes\cite{LionsMagenes} under minimal assumptions on the regularity of the
coefficient functions. A classical analysis regarding equations of the
type \eqref{eq: 2.4} is attributed to Dafermos\cite{dafermos1970abstract}. Here, the
domain can be the whole space, but the requirements on the initial
conditions exclude plane waves. Blazek \textit{et
al.}\cite{blazek2013mathematical} proved the same result for systems
of equations. Kim\cite{kim1994local} obtained existence and
uniqueness of solutions using Friedrichs mollifier techniques assuming
that the coefficient functions are smooth in space and time while
allowing plane-wave initial conditions. Kim's analysis
also motivates the development of a microlocal analogue of the
correspondence principle\cite{paulino2003} in a parametrix
construction starting from plane-wave initial values.

Following B\'{e}cache \textit{et al.}\cite{becache2005mixed} we rewrite the system with relaxation based on
a generalized Zener solid in first-order partial-differential
form. They obtained well-posedness under minimal assumptions
on the regularity of the coefficient functions, again, excluding
plane-wave initial conditions. Here, we study well-posedness of
solutions of such a system with constant (time- and space-independent)
coefficients in the whole space through a plane-wave synthesis and
analysis. This is motivated by the calculations carried out by
Richards\cite{richards1984wave} pertaining to plane-wave reflection
in bimaterials with relaxation. Richards observed that in a
configuration of two distinct homogeneous isotropic viscoelastic
solids separated by a plane interface, at particular scattering angles
plane waves will exhibit an exponentially growing behavior. We will
study the stability of solutions in a generalized Zener solid with an
explicit dependence on the parameters controlling the
relaxation. 

 Solem \textit{et al.}\cite{solem} considered one-dimensional linear hyperbolic systems with a stable relaxation term of rank 1 and pointed out a connection between stability properties of such systems and the theory describing general properties of polynomial roots. In particular, it was shown in \cite{solem} that strictly hyperbolic relaxation systems are linearly stable if and only if the roots of the homogeneous and equilibrium
characteristic polynomials interlace on the imaginary axis. In Sec. \ref{Sec.4} we invoke the Routh-Hurwitz theorem to determine the number of roots of the characteristic polynomial in the right half-plane and mention in Sec. \ref{Sec.5} how the location and multiplicity of roots influence stability.

In \cite{becache2005mixed} B\'{e}cache \textit{et al.} defined the following quantity as the energy of the model 

\begin{equation}\label{eq: 1.1}
E(q,\sigma,t)=\frac{1}{2}\left\Vert \dot{q}\right\Vert _{\rho}^{2}+\frac{1}{2}\left\Vert \epsilon(q)\right\Vert _{{\bf C}}^{2}+\frac{1}{2}\left\Vert s\right\Vert _{({\bf D-C})^{-1}}^{2}
\end{equation}
The sum of the first two terms in \eqref{eq: 1.1} corresponds to the standard energy in the
purely elastic case and the final term is the norm of the difference between viscoelastic and
elastic stresses. It turns out that in the absence of the source term the energy decreases in time if the absorption condition holds, i.e. ${\bf D-C}$ is positive definite where $\bf D$ and $\bf C$ are two symmetric tensors of order four that define the constitutive law (see Ref. \cite{becache2005mixed} for details and notation). In Sec. \ref{Sec.5} we perform a similar analysis in Fourier space and comment on the conditions of energy dissipation.

\section{Memory kernels and relaxation}\label{Sec.2}

 For an arbitrary point $x\in\mathbb{R}^d$ in the medium let the vector-valued
displacement of the point from its position in an undeformed state be $q (x,t)$, let $\sigma_{ij}(x,t)$
be the stress tensor with $(\nabla\cdot\sigma)_i=\sum_{j=1}^d{\partial \sigma_{ij}\over\partial {x_j} }$, let $F(x,t)$ represent the external
forces per unit volume and $\rho(x)$ denote the density. 
The description of wave
propagation in a general medium is expressed by the equation of motion
\begin{equation}\label{eq: 2.1}
\rho \ddot{q}_{i}=(\nabla\cdot\sigma)_i+  F_i, \ \ \ \ \ \ \ \ \ \ \ \ \  i=1,\dots, d
\end{equation}
which follows from the conservation of linear momentum. 

The so-called Zener or standard
linear solid model provides the most general linear constitutive law
between the stress, strain and their rates of change 
\begin{equation}\label{eq: 2.2} 
\sigma+\tau_{\sigma}\dot{\sigma}=M_{R}(\epsilon(q)+\tau_{\epsilon}\dot{\epsilon}(q))
\end{equation}
relating them by three parameters: the deformation modulus $M_{R}$,
the stress relaxation time $\tau_{\sigma}$ and the strain relaxation
time $\tau_{\epsilon}$ \cite{liu1976}.

Rewriting Eq. \eqref{eq: 2.2} in the following equivalent form 
\[
\partial_{t}\left(e^{t/\tau_{\sigma}}\sigma\right)=M_{R}\frac{\tau_{\epsilon}}{\tau_{\sigma}}\partial_{t}\left(e^{t/\tau_{\sigma}}\epsilon\right)+\frac{M_{R}(\tau_{\sigma}-\tau_{\epsilon})}{\tau_{\sigma}^{2}}e^{t/\tau_{\sigma}}\epsilon
\]
and integrating it choosing the initial condition $\sigma_{0}=M_{R}\frac{\tau_{\epsilon}}{\tau_{\sigma}}\epsilon_{0}$ results in the stress-strain relation 
\begin{equation}\label{eq: 2.3}
\sigma=M_{R}\frac{\tau_{\epsilon}}{\tau_{\sigma}}\epsilon+\frac{M_{R}(\tau_{\sigma}-\tau_{\epsilon})}{\tau_{\sigma}^{2}}\int_{0}^{t}e^{-(t-s)/\tau_{\sigma}}\epsilon(s)ds.
\end{equation}
The first term on the right-hand side of \eqref{eq: 2.3} represents Hooke's law and
the second term indicates that the stress at any given instance depends
upon the strain at all preceding times. The idea
that stress depends both on the present and past value of strain is
attributed to Boltzmann. Early contributions are also due to Maxwell,
Kelvin and Voigt\cite{renardy}. 

 Dividing both sides of \eqref{eq: 2.1} by the density and taking the divergence results in 
\[\phi_{tt}=\nabla\cdot\left({1\over \rho}\nabla\cdot\sigma\right)+f\] where $f=\nabla\cdot\left(F/\rho\right)$ and $\phi=\nabla\cdot q$ are scalar-valued functions. 
Substituting Eq. \eqref{eq: 2.3} into to the above equation and remembering that the strain tensor and the displacement vector satisfy $\epsilon_{ij}=1/2(\partial q_i/\partial x_j+\partial q_j/\partial x_i)$ we arrive at the second-order integro-differential
equation modeling viscoelastic motion 
\begin{equation}\label{eq: 2.4}
\phi_{tt}=\nabla\cdot\left(c^{2}(x)\nabla\phi\right)+\int_{0}^{t}\nabla\cdot\left(a(x) e^{-(t-s)/\tau_{\sigma}}\nabla\phi(x,s)\right)ds+f ,
\end{equation}
where $c^{2}(x)=\frac{2\mu+\lambda}{\rho}\frac{\tau_{\epsilon}}{\tau_{\sigma}}$,
$a(x)=\frac{2\mu+\lambda}{\rho}\frac{(\tau_{\sigma}-\tau_{\epsilon})}{\tau_{\sigma}^{2}}$
and deformation modulus is written in terms of the Lame parameters,
that is $M_{R}=2\mu+\lambda$. 

When an elastic body is under the effect of hydrostatic pressure, i.e. when a pressure of the same magnitude acts on every unit area on the surface of the body,  both the strain and stress tensors are determined by their diagonal components. In fact, if $p(x,t)$ is the pressure field, then $\sigma_{ij}=-p\delta_{ij}$. In this case a derivation similar to the one carried out above yields a scalar wave equation  for $p=-1/3\mbox{tr}( \sigma)$ describing the propagation of acoustic
waves in a viscoelastic fluid \cite{petrov} (see also \cite{carcione} for the derivation of a scalar wave equation for the trace of the strain tensor or the dilatation).

Quite often a combination of weightless springs and dashpots filled
with viscous fluids is used as a good mechanical model that describes
anelastic phenomena and the behavior of a variety of materials. A spring
and a dashpot connected in series yield the Maxwell model, while being
connected in parallel give the Kelvin-Voigt model. These models can
be obtained from the Zener model in \eqref{eq: 2.2} by taking the limits $\tau_{\epsilon}\to\infty$
and $\tau_{\sigma}\to0$, respectively. 

The generalized Zener model consists of a number of Zener elements combined
in parallel and takes into account multiple relaxation times. The
total stress acting on the system is the sum of the stresses experienced
by each element $\sigma=\sum_{i=1}^{k}\sigma_{i}$. Denoting the deformation
moduli and relaxation times by
\[
M_{Ri}=\frac{E_{1i}E_{2i}}{E_{1i}+E_{2i}},\tau_{\sigma i}=\frac{\eta_{i}}{E_{1i}+E_{2i}}=\frac{1}{b_{i}},\tau_{\epsilon i}=\frac{\eta_{i}}{E_{2i}},i=1,2,\dots,k ,
\]
where $E_{1i},E_{2i}$ are the Young moduli of the springs in the
$i$th element and $\eta_{i}$ is the viscosity of the corresponding
dashpot we arrive at the generalization of Eq. \eqref{eq: 2.4} with $c^{2}=\sum_{i=1}^{k}M_{Ri}\tau_{\epsilon i}\tau_{\sigma i}^{-1}$
\begin{equation}\label{eq: 2.5}
\phi_{tt}=\nabla\cdot\left(c^{2}( x)\nabla\phi\right)+\sum_{i=1}^{k}\int_{0}^{t}\nabla\cdot\left(a_{i}( x)e^{-b_{i}( x)(t-s)}\nabla\phi( x,s)\right)ds+ f .
\end{equation} \\ \\ \\  \\
%For the sake of mathematical completeness one could consider a further
%generalization by introducing explicit time dependence in $c$ and
%$a$: 
%\begin{equation}\label{eq: 2.5}
%\phi_{tt}=\nabla\cdot\left(c^{2}( x,t)\nabla\phi\right)+\sum_{i=1}^{k}\int_{0}^{t}\nabla\cdot\left(a_{i}(x,s)e^{-b_{i}(x)(t-s)}\nabla\phi(x,s)\right)ds+f.
%\end{equation}
We will assume that \\
\begin{enumerate}
\item[] {\bf [A1]} $c^{2}$ is positive bounded away from zero, $b_{i}>0$ are pairwise distinct and no  sign  condition is imposed on the coefficients $a_{i}\ne0$, unless otherwise stated. \\
\end{enumerate} 

Since \eqref{eq: 2.5} is linear, by considering the difference of solutions we can study the effect of the sufficiently regular external force separately with zero initial conditions. Therefore, in what follows, we put $f=0$. 

%\begin{remark}\label{remark 2.1}
%In the case of constant relaxation times and varying
%density $\rho(x,t)$ the
%conservative structure of the memory term breaks down and we can derive a non-conservative version of Eq. \eqref{eq: 2.1} in full generality: 
%\[
%\phi_{tt}=c^{2}(x,t)\phi_{xx}+\sum_{i=1}^{k}\int_{0}^{t}a_{i}(x,s)e^{-b_{i}(x)(t-s)}\phi_{xx}(x,s)ds
%\]
%This integro-differential equation can be rewritten in first-order partial-differential form \eqref{eq: 2.6}
%by introducing the auxiliary functions\cite{markowich} $u=\phi_{x}$, $v=-\phi_{t}$ and
%$w_{i}=\int_{0}^{t}a_{i}(x,s)e^{-b_{i}(x)(t-s)}\phi_{xx}(x,s)ds$,
%which are different from those used below for the conservative form in
%Eq. \eqref{eq: 2.5}.
%\end{remark}

\subsection*{Initial value problem}

Let $d\geq1$ be the space dimension and $x=(x_{1},x_{2},\dots,x_{d})\in\mathbb{R}^{d}$
be the space and $t\in\mathbb{R}$ the time variables. It is convenient
to formulate the equation of motion derived in the previous section
as a Cauchy initial value problem 
\begin{eqnarray}
 \mathcal{L}(\partial_{t},\nabla)U=\partial_{t}U+\sum_{j=1}^{d}A_{j}(x,t)\partial_{x_{j}}U+B(x,t)U &=& 0 ,  \nonumber \\
 U(x,0) &=& U_{0} , \label{eq: 2.6}
\end{eqnarray}
where $U:\mathbb{R}^{d}\times(0,T)\to\mathbb{R}^{n}$ is the unknown
vector, $A_{j},B:\mathbb{R}^{d}\times(0,T)\to\mathbb{M}^{n\times n}$
are matrix coefficients with $n=kd+d+1\geq3$ being the size of the
system, and the initial datum $U_{0}:\mathbb{R}^{d}\to\mathbb{R}^{n}$
is given in a suitable function space.  Using the substitution 
\begin{align*}
 & u = -\phi_t(x,t) ,\\
& v=c^2(x)\nabla\phi+\sum_{i=1}^{k} \int_{0}^{t}a_{i}(x)e^{-b_{i}(x)(t-s)}\nabla \phi(x,s)ds,\\
& w_{i} = -a_{i}(x)\nabla\phi+ b_{i}(x)\int_{0}^{t}a_{i}(x)e^{-b_{i}(x)(t-s)}\nabla\phi(x,s)ds,\ \ \ i=1,2,\dots,k , \\ 
& v=(v_1, v_2,\dots, v_d), \ \ w_i=(w_{i_1},w_{i_2}, \dots, w_{i_d})
 \end{align*}
Eq. \eqref{eq: 2.5} can be recast as a system 
\begin{eqnarray*}
 u_{t}+\nabla\cdot v &=& 0 ,\\
 v_{t}+c^{2}(x)\nabla u+\sum_{i=1}^{k}w_{i} &=& 0 ,\\
 \left(w_{i}\right)_{t}-a_{i}(x)\nabla u+b_{i}(x)w_{i} &=& 0 ,\ \ \ i=1,2,\dots,k ,
\end{eqnarray*}
which can be written as 
\[
\left(\begin{array}{c}
u\\
v\\
w_{1}\\
\vdots\\
w_{k}
\end{array}\right)_{t}+\underset{A(\nabla)}{\underbrace{\left(\begin{array}{ccccc}
0 & \nabla\cdot & 0 & \cdots & 0\\
c^{2}\nabla & 0 & 0 & \cdots & 0\\
-a_{1}\nabla & 0 & 0 & \cdots & 0\\
\vdots & \vdots & \vdots & \vdots & \vdots\\
-a_{k}\nabla & 0 & 0 & \cdots & 0
\end{array}\right)}}\left(\begin{array}{c}
u\\
v\\
w_{1}\\
\vdots\\
w_{k}
\end{array}\right)+\underset{B\left(u,v,w_{1},\dots,w_{k}\right)^{T}}{\underbrace{\left(\begin{array}{c}
0\\
\sum_{i=1}^{k}w_{i}\\
b_{1}w_{1}\\
\vdots\\
b_{k}w_{k}
\end{array}\right)}}=0.
\]
Expanding $A(\nabla)$ as $A(\nabla)=\sum_{j=1}^{d}A_{j}(x,t)\partial_{x_{j}}$
we arrive at \eqref{eq: 2.6} with $U=(u,v_1,\dots, v_d,w_{1_1},w_{1_2},\dots,w_{k_d})$
and $U_{0}=(u^0,v_1^0,\dots, v_d^0,w_{1_1}^0,w_{1_2}^0,\dots,w_{k_d}^0)$. 

One can recover $\phi(x,t)$ by first noting that 
\begin{align*}
&\phi_{t}(x,t)=-u,\\ & \nabla\phi(x,t) = \frac{v-\sum_{i=1}^{k}\frac{w_{i}}{b_{i}}}{c^{2}+\sum_{i=1}^{k}\frac{a_{i}}{b_{i}}}\end{align*} and then using the fundamental theorem for
gradients. The condition $c^{2}+\sum_{i=1}^{k}\frac{a_{i}}{b_{i}}>0$ is motivated on physical grounds (see assumption {\bf [A2]} and Remark \ref{remark 4.1}) and $b_{i}>0$ by assumption {\bf [A1]}, so $\nabla\phi$ in the second equality is well-defined.

\section{Well-posedness of the initial value problem}\label{Sec.3}

For the rest of the paper we shall consider the constant-coefficient systems, i.e. assume that
$A_{j},B$ are independent of $(x,t)$ and denote the principal part
of $\mathcal{L}$ given in \eqref{eq: 2.6} by 
\[
L=\partial_{t}+\sum_{j=1}^{d}A_{j}\partial_{x_{j}}.
\]
The Fourier transform of \eqref{eq: 2.6} in the spatial directions gives 
\begin{equation}\label{eq: 3.1}
\hat{U}_{t}+i\sum_{j=1}^{d}\xi_{j}A_{j}\hat{U}+B\hat{U}=0,\ \hat{U}(\xi,0)=\hat{U}_{0}
\end{equation}
where $\xi=(\xi_{1},\xi_{2},\dots,\xi_{d})\in\mathbb{R}^{d}$ is a
vector dual to $x$. Using the notation
$A(\xi)=\sum_{j=1}^{d}\xi_{j}A_{j}$ we can write the solution of this
ordinary differential equation as
$\hat{U}(\xi,t)=e^{-t\left(B+iA(\xi)\right)}\hat{U}_{0}(\xi)$.  When
$U_{0}\in H^{s}\left(\mathbb{R}^{d}\right)^{n}$, by taking the inverse
Fourier transform one can show that the Cauchy problem \eqref{eq: 2.6} admits a
continuous solution
\begin{equation}\label{eq: 3.2}
U(x,t)=\frac{1}{(2\pi)^{d/2}}\int_{\mathbb{R}^{d}}e^{ix\cdot\xi}\hat{U}(\xi,t)d\xi
\end{equation}
with values in $H^{s}$ if 
\[
\sup_{\xi\in\mathbb{R}^{d},0\leq t\leq T}\left\Vert e^{-t\left(B+iA(\xi)\right)}\right\Vert <\infty
\]
which is equivalent  (see, for example, proposition 2.I.1 in \cite{rauch}) to writing 
\begin{equation}\label{eq: 3.3}
\sup_{\xi\in\mathbb{R}^{d}}\left\Vert e^{-iA(\xi)}\right\Vert <\infty.
\end{equation}
Throughout this paper, we will use the matrix norm $\left\Vert
M\right\Vert =\sup_{|x|=1}|Mx|$ induced by the Euclidean norm. Notice that the property \eqref{eq: 3.3} does not depend on time once $t\ne 0$ since $tA(\xi)=A(t\xi)$. We can also absorb the minus sign by virtue of the change $\xi\to-\xi$. \\

\begin{definition}\label{def. 3.1}
The operator $L$ is called hyperbolic if the corresponding symbol
$A(\xi)$ satisfies \eqref{eq: 3.3}.
\end{definition} \\

\begin{proposition}\label{prop. 3.1}
Assume that {\bf [A1]} holds. The
matrix $A(\xi)$ is uniformly diagonalizable with real eigenvalues:
There exists $P(\xi)$ such that $P(\xi)A(\xi)P^{-1}(\xi)$ is diagonal
and real for all $\xi\in\mathbb{R}^{d}$ and
\[
   \sup_{\xi\in \mathbb{S}^{d-1}}\left\Vert P^{-1}(\xi)\right\Vert
        \left\Vert P(\xi)\right\Vert <\infty .
\]
\end{proposition}

\begin{proof} A simple computation shows that the characteristic
equation of $A$ splits as 
\[
\lambda^{kd+d-1}\left(\lambda^{2}-c^{2}|\xi|^{2}\right)=p_{1}^{kd+d-1}(\lambda,\xi)p_{\pm}(\lambda,\xi) .
\] where $p_1(\lambda,\xi)=\lambda$ and $p_{\pm}(\lambda,\xi)=\lambda^{2}-c^{2}|\xi|^{2}$. 
Observe that $p_{\pm}$ and $p_1$ are homogeneous polynomials in $\lambda, |\xi|$
and $\lambda$, respectively, with real and simple roots, and they have
no common root for $\xi\in\mathbb{R}^{d}\backslash\{0\}$. Let $$E_{j}=\prod_{i\ne j}\frac{A-\lambda_{i}I_{n}}{\lambda_{j}-\lambda_{i}}, \ \ \ \ \mbox{for } j=1, 2, 3 \mbox{ and } n=kd+d+1$$
with $\lambda_1=0$ and $\lambda_{2,3}=\pm c |\xi|$. One can check that $E_j$'s are mutually orthogonal and complete in the sense that $E_{i}E_{j}=\delta_{ij}E_{j}$
and $\sum_{j=1}^{3}E_{j}=I_{n}$ and verify that the following decomposition takes place  \[
A(\xi)=\sum_{j=1}^{3}\lambda_{j}E_{j}. 
\]
%where the eigenvalues are as above and $$E_{j}=\prod_{i\ne j}\frac{A-\lambda_{i}I_{n}}{\lambda_{j}-\lambda_{i}}$$
%are mutually orthogonal projections satisfying $E_{i}E_{j}=\delta_{ij}E_{j}$
%and $\sum_{j=1}^{3}E_{j}=I_{n}$. 

Next we define a positive-definite matrix $H(\xi)=\sum_{j}E_{j}^{T}E_{j}$
which admits a unique square root. Then, since $A(\xi)^{T}=\sum_{j=1}^{3}\lambda_{j}E_{j}^{T}$,
it follows that $H(\xi)A(\xi)=A(\xi)^{T}H(\xi)$ which implies that
$H^{1/2}(\xi)A(\xi)H^{-1/2}(\xi)$ is symmetric and diagonalizable
in an orthonormal basis. Hence $A(\xi)=P^{-1}(\xi)D(\xi)P(\xi)$ where
$D(\xi)$ is diagonal with real eigenvalues and $P(\xi)=O(\xi)H^{1/2}(\xi)$
with an orthogonal matrix $O(\xi)$. 

It remains to show that $P(\xi)$ is uniformly bounded. We note that
\[
|y|^{2}=\left|\sum_{j=1}^{3}E_{j}y\right|^{2}\leq3\sum_{j=1}^{3}|E_{j}y|^{2}=3|H^{1/2}y|^{2}
\]
and thus $\left\Vert H^{-1/2}(\xi)\right\Vert \leq\sqrt{3}$. 

Using the Lagrange multiplier method with the constraint $|x|=1$ or calculating the largest eigenvalue of $E_j^TE_j$ we find that
\begin{eqnarray*}
 &  & \left\Vert E_{1}\right\Vert ^{2}=1+\frac{1}{c^{4}}\sum_{i=1}^{k}a_{i}^{2} ,\\
 &  & \left\Vert E_{2}\right\Vert ^{2}=\left\Vert E_{3}\right\Vert ^{2}=\frac{1+c^{2}}{4}\left(\frac{1}{c^{2}}+1+\frac{1}{c^{4}}\sum_{i=1}^{k}a_{i}^{2}\right) ,
\end{eqnarray*}
are independent of $\xi$, and therefore 
\[
|H^{1/2}y|^{2}=\sum_{j=1}^{3}|E_{j}y|^{2}\leq\left(\left\Vert E_{1}\right\Vert ^{2}+\left\Vert E_{2}\right\Vert ^{2}+\left\Vert E_{3}\right\Vert ^{2}\right)|y|^{2}\leq\frac{C^{2}}{3}|y|^{2}.
\]
We conclude that the diagonalization is well-conditioned because 
\[
\left\Vert P^{-1}(\xi)\right\Vert \left\Vert P(\xi)\right\Vert =\left\Vert H^{-1/2}(\xi)\right\Vert \left\Vert H^{1/2}(\xi)\right\Vert \leq C
\]
independently of $\xi$. \end{proof}  \\

\begin{definition}\label{def. 3.2}
The Cauchy problem for a constant coefficient
operator $\mathcal{L}$ is weakly (strongly) well-posed if for any
initial data $U_{0}\in H^{s}\left(\mathbb{R}^{d}\right)$ with $s>0$
($s=0$), there is a unique solution $U(t)\in\mathcal{C}\left(\mathbb{R}^{+},H^{s}(\mathbb{R}^{d})\right)$
that satisfies 
\[
\left\Vert U(t)\right\Vert _{L^{2}\left(\mathbb{R}^{d}\right)}\leq Ke^{\alpha t}\left\Vert U_{0}\right\Vert _{H^{s}\left(\mathbb{R}^{d}\right)},\ \ \ t\geq0 ,
\]
with $K>0$ and $\alpha\in\mathbb{R}$ independent of time. 
\end{definition} \\

\begin{lemma}\label{lem. 3.1}(Strang, \cite{strang2})
If $\left\Vert e^{tA}\right\Vert \leq C$ for $t\geq 0$, then $\left\Vert e^{t(A+B)}\right\Vert \leq C e^{tC\left\Vert B\right\Vert}$.
\end{lemma}
\begin{proof}
This is an exponential analogue of another lemma due to Strang \cite{strang1} which states that if $\left\Vert M^n\right\Vert \leq C$ for $n\geq 0$, then $\left\Vert (M+R)^n\right\Vert \leq Ce^{nC\left\Vert R \right\Vert}$. Setting $M=e^{\varepsilon A}$ and $R=e^{\varepsilon(A+B)}-M$ with sufficiently small $\varepsilon$ we have $\left\Vert M^n\right\Vert \leq C$ for $n\geq 0$ and hence  \[\left\Vert (M+R)^n\right\Vert = \left\Vert e^{n\varepsilon(A+B)}\right\Vert \leq Ce^{nC\left\Vert R \right\Vert} \] Let $n$ tend to infinity, while keeping $t=n\varepsilon$ fixed. In this limit we have $n R \to t B$ and the lemma follows. 
\end{proof} \\

\begin{theorem}\label{th. 3.1}
Assume that {\bf [A1]} holds. The
operator $L$ is hyperbolic and the Cauchy problem for a constant
coefficient operator $\mathcal{L}$ is strongly well-posed.
\end{theorem}

\begin{proof} By Proposition \ref{prop. 3.1}, we have that for all $\xi\in\mathbb{R}^{d}$
and $t\geq0$
\begin{eqnarray*}
\left\Vert e^{itA(\xi)}\right\Vert  & \leq & \left\Vert P^{-1}e^{itD}P\right\Vert \leq C ,
\end{eqnarray*}
 where $D$ is diagonal with real entries, $e^{itD}$ is unitary and
therefore leaves the matrix norm invariant. Hence $L$ is hyperbolic. 

Using \eqref{eq: 3.2}, Parseval's relation and hyperbolicity of $L$ we obtain the
following estimate 
\begin{align*}
 \left\Vert U(t)\right\Vert _{L^{2}\left(\mathbb{R}^{d}\right)} & = \left\Vert e^{-t\left(B+iA(\xi)\right)}\hat{U}_{0}(\xi)\right\Vert _{L^{2}\left(\mathbb{R}^{d}\right)}
\\
& \leq\left\Vert e^{-t\left(B+iA(\xi)\right)}\right\Vert \left\Vert \hat{U}_{0}(\xi)\right\Vert _{L^{2}\left(\mathbb{R}^{d}\right)}\leq Ce^{tC\left\Vert B\right\Vert }\left\Vert U_{0}\right\Vert _{L^{2}\left(\mathbb{R}^{d}\right)} \nonumber
\end{align*}
completing the claim. The last inequality follows from Lemma \ref{lem. 3.1}. Note that since $A$ and $B$ do not commute,
it does not hold that $\left\Vert e^{t\left(A+B\right)}\right\Vert =\left\Vert e^{tA}e^{tB}\right\Vert $
for $t>0$. \end{proof}  \\

\begin{remark}\label{remark 3.1}
1. If $L$ is hyperbolic and $U_{0}\in
H^{s}\left(\mathbb{R}^{d}\right)^{n}$, then application of Gronwall's
inequality shows that there is a continuous solution with values in
$H^{s}$ if one has a variable-coefficient lower order term $B(x)\in
L^{\infty}\left(\mathbb{R}^{d}\right)$.  In this case the Cauchy
problem for $L+B$ is also strongly well-posed.  Hyperbolicity and
well-posedness is a property of $A$ alone.

2. In the notation of Def. \ref{def. 1.1} we have $\lambda_1=-c|\xi|$, $\lambda_2=0$ and $\lambda_3=c|\xi|$ with $l_1=1$, $l_2=kd+d-1$ and $l_3=1$. Operator $L$ is constantly hyperbolic, that is, the symbol $A(\xi)$
is diagonalizable with real eigenvalues and the algebraic multiplicities
of eigenvalues remain constant as $\xi$ ranges along  $\mathbb{S}^{d-1}.$
Strict or constant hyperbolicity implies hyperbolicity\cite{serre2007}. 

In \cite{metivier} M\'{e}tivier provided a few examples of systems satisfying the block structure condition, including the equations of linear elasticity. Our results show that this important class of systems can be enlarged by the generalized Zener model of viscoelasticity.

3. The matrices $A_{j}$ do not commute, i.e. $A_{j}A_{i}\ne A_{i}A_{j}$
for $i\ne j$. Hence they cannot be simultaneously diagonalized and
Eq. \eqref{eq: 3.1} cannot be transformed to a system consisting of $n$ uncoupled
scalar equations. 

 4. Eq. \eqref{eq: 3.1} can be viewed as a linearization of a system with a non-linear source term $Q(U)$ about a
constant state in equilibrium. Typically, the source term is divided by a small parameter that determines the rate of relaxation towards equilibrium. To ensure  the existence of a well-behaved zero relaxation limit, Yong \cite{yong} introduced the so-called stability criterion which necessitates that there is $C(U)>0$ such that \[
\left\Vert e^{\delta Q_U(U)+iA(\xi)}\right\Vert \leq C(U)
\] for all $\delta \geq 0$, $\xi \in\mathbb{R}^d$ with $\{U: Q(U)=0\}\ne \emptyset$. Here, $Q_U(U)$ denotes the Jacobian matrix of the source term. This criterion is somewhat stronger than the hyperbolicity condition and reduces to that when $\delta=0$ (cf. Eq. \eqref{eq: 3.3} and the inequality preceding it). 
\end{remark}

\section{Plane-wave analysis}\label{Sec.4}

Waves at a sufficiently large distance from the source behave locally
like plane waves. This motivates one to study the behavior of plane
waves as possible growth modes in the system under consideration. \\
\begin{theorem}\label{th. 4.1}
Let $A_{j}$, $B$ be constant-coefficient matrices and $d\geq1$.
The eigenvalues of $\Phi(i\xi)=-\left(B+iA(\xi)\right)$ are roots of the characteristic polynomial
\begin{equation}\label{eq: 4.1}
\tilde{p}(\lambda,\xi_{1},\dots,\xi_{d})=p(\lambda,|\xi|)\lambda^{d-1}\prod_{i=1}^{k}(\lambda+b_{i})^{d-1}
\end{equation}
where $p(\lambda,|\xi|)$ is the characteristic polynomial corresponding
to the system derived from the one-dimensional wave equation with
$k$ memory terms and $\xi\in\mathbb{R}$ replaced by $|\xi|=\sqrt{\xi_{1}^{2}+\dots+\xi_{d}^{2}}\in\mathbb{R}$:
\[
p(\lambda,\xi)=(-1)^{k}\left(\lambda^{2}+c^{2}\xi^{2}+\xi^{2}\sum_{i=1}^{k}\frac{a_{i}}{\lambda+b_{i}}\right)\prod_{i=1}^{k}(\lambda+b_{i})
\]
\end{theorem} 

 \begin{proof}
Perform the following similarity transformation: pre-multiply  $\lambda I_{kd +d +1} - \Phi(i\xi)$ by the block diagonal matrix 
\[S=\left(\begin{array}{ccccc}
|\xi| & 0 & \dots & \dots & 0\\
0 & \Xi_1 & 0 & \dots & 0\\
0 & 0 & \ddots & \ddots & \vdots\\
\vdots & \vdots & \ddots & \ddots & 0\\
0 & 0 & \dots & 0 & \Xi_{k+1}
\end{array}\right),\ \mbox{ where } \ \Xi_j=\left(\begin{array}{ccccc}
\xi_{1} & \xi_{2} & \xi_{3} & \dots & \xi_{d}\\
0 & \xi_{2} & \xi_{3} & \dots & \xi_{d}\\
0 & 0 & \ddots & \ddots & \vdots\\
\vdots & \vdots & \ddots & \ddots & \xi_{d}\\
0 & 0 & \dots & 0 & \xi_{d}
\end{array}\right)\]
are identical for all $1\leq j \leq k+1,$ and post-multiply by its inverse $S^{-1}$. Successively develop the resulting determinant by the columns containing a single non-zero element thereby accounting for the factor $\lambda^{d-1}\prod_{i=1}^{k}(\lambda+b_{i})^{d-1}$ in Eq. \eqref{eq: 4.1}. \end{proof} \\

Without loss of generality, we assume that $0<b_{1}<b_{2}<\dots<b_{k}$
and consider $g(\lambda)=(-1)^{k}p(\lambda,\xi)$. Since the
characteristic polynomial in higher dimensions splits as in \eqref{eq: 4.1},
it suffices to analyze the roots of $p(\lambda,\xi)$. \\

\begin{proposition}[All $a$'s are negative]\label{prop. 4.1} Let $d=1$, $\xi\ne0$
and $a_{i}<0$ for all $1\leq i\leq k$, then i) if $g(0)>0$,
all eigenvalues of $\Phi(i\xi)$ have negative real parts; ii) if
$g(0)=0$, then one eigenvalue is zero and the rest have negative
real parts; iii) if $g(0)<0$, then only one eigenvalue of $\Phi(i\xi)$
is positive and all other eigenvalues have negative real parts. 
\end{proposition}

\begin{proof} If $\xi=0$, the eigenvalues are $-b_{i}<0$ and 0
with multiplicity 2. 

If $\xi\ne0$, the function $g(\lambda)/\prod_{i=1}^{k}(\lambda+b_{i})$
has $k$ simple poles at $-b_{i}$ and since all $a_{i}$'s have the
same sign, there are $k-1$ real roots $r_{i}$ of $p(\lambda,\xi)$
between these poles. 

i) Assume that $g(0)>0$, then a further real root lies between $0$
and $-b_{1}$, as follows from 
\[
g(-b_{1})g(0)=\xi^{2}a_{1}\prod_{i\ne1}(b_{i}-b_{1})g(0)<0
\]
and the Intermediate value theorem. The function $g(\lambda)$ can
now be factored as 
\[
(\lambda^{2}+\alpha\lambda+\beta)\prod_{i=1}^{k}(\lambda-r_{i})=0.
\]
The coefficients $\alpha,\beta$ are real since $r_{i}$'s are real
for all $1\leq i\leq k$. Denote by $r_{k+1}$ and $r_{k+2}$ the
two roots (real or complex conjugate) of $\lambda^{2}+\alpha\lambda+\beta$,
then by Vieta's theorem, 
\begin{align*}
& r_{k+2}+r_{k+1} = -\sum_{i=1}^{k}b_{i}-\sum_{i=1}^{k}r_{i}<0 ,\\ 
 & r_{k+2}\cdot r_{k+1} = g(0)/\prod_{i=1}^{k}|r_{i}|>0 .
\end{align*}
If $r_{k+1}$ and $r_{k+2}$ are complex conjugate, then $\Re(r_{k+1})=\Re(r_{k+2})<0$.
If $r_{k+1}$ and $r_{k+2}$ are real, then $r_{k+1}<0$ and $r_{k+2}<0$.
The same result was obtained  in\cite{markowich}.

ii) If $g(0)=0$, then in addition to a real root $r_{i}$ between
each consecutive $-b_{i}$'s, there is a zero eigenvalue since the
constant term in $p(\lambda,\xi)$ is absent and therefore one can
factor out $\lambda$:
\[
\lambda(\lambda^{2}+\alpha\lambda+\beta)\prod_{i=1}^{k-1}(\lambda-r_{i})=0.
\]
By Vieta's theorem 
\begin{align*}
 & r_{k+2}+r_{k+1} = -\sum_{i=1}^{k}b_{i}-\sum_{i=1}^{k-1}r_{i}<0,\\
 & r_{k+2}\cdot r_{k+1} = g'(0)/\prod_{i=1}^{k-1}|r_{i}|=-\xi^{2}\left(\sum_{i=1}^{k}\frac{a_{i}}{b_{i}^{2}}\right)\prod_{i=1}^{k}b_{i}/\prod_{i=1}^{k-1}|r_{i}|>0
\end{align*}
where $r_{k+1}$ and $r_{k+2}$ are roots (real or complex conjugate)
of $\lambda^{2}+\alpha\lambda+\beta$ and we used $g(0)=0$ in the
expression for $g'(0)$. As in case i) above, $r_{k+1}$ and $r_{k+2}$
have negative real parts. 

iii) Now assume that $g(0)<0$. Since there are $k-1$ real roots
$r_{i}$ between $k$ simple poles $-b_{i}$, $g(\lambda)$ can be
written as 
\[
\left(\lambda^{3}+\alpha\lambda^{2}+\beta\lambda+\gamma\right)\prod_{i=1}^{k-1}(\lambda-r_{i})=0
\]
By Vieta's theorem
\begin{align}
 & r_{k+1}+r_{k+2}+r_{k} = -\sum_{i=1}^{k}b_{i}-\sum_{i=1}^{k-1}r_{i}<0, \nonumber \\
 & r_{k+1}\cdot r_{k+2}\cdot r_{k} = (-1)^{k+2}g(0)/\prod_{i=1}^{k-1}r_{i}=-g(0)/\prod_{i=1}^{k-1}|r_{i}|>0 \label{eq: 4.2} 
\end{align}
where $r_{k},r_{k+1},r_{k+2}$ are roots of the cubic equation. An
algebraic equation of an odd degree and real coefficients must posses
at least one real root. Eq. \eqref{eq: 4.2} implies that this root is positive.
The other two roots of the cubic equation have negative real parts.
\end{proof} \\

\begin{proposition}[All $a$'s are positive]\label{prop. 4.2} Let $d=1$, $\xi\ne0$
and $a_{i}>0$ for all $1\leq i\leq k$, then two eigenvalues
of $\Phi(i\xi)$ have positive real parts and the others are real and
negative. 
\end{proposition}

\begin{proof} If $\xi=0$, the eigenvalues are $-b_{i}<0$ and zero
(two-fold). 

If $\xi\ne0$, then the $k-1$ real roots $r_{i}$ of $p(\lambda,\xi)$
strictly interlace $-b_{i}$ for $1\leq i\leq k$. By the Intermediate
value theorem there is also a root to the left of $-b_{k}=-\max_{i}b_{i}$
since $\lim_{\lambda\to-\infty}p(\lambda,\xi)=+\infty$ and $p(-b_{k},\xi)=-\xi^{2}a_{k}\prod_{i\ne k}(b_{k}-b_{i})<0$.
Thus for some real $\alpha,\beta$ we have the factorization 
\[
(\lambda^{2}+\alpha\lambda+\beta)\prod_{i=1}^{k}(\lambda-r_{i})=0.
\]
By Vieta's theorem the roots $r_{k+1},r_{k+2}$ of $\lambda^{2}+\alpha\lambda+\beta$
satisfy 
\begin{align*}
& r_{k+2}+r_{k+1} = -\sum_{i=1}^{k}b_{i}-\sum_{i=1}^{k}r_{i}>0,\\
 & r_{k+2}\cdot r_{k+1} = g(0)/\prod_{i=1}^{k}|r_{i}|>0
\end{align*}
where $g(0)>0$, since $a_{i}>0$ for all $1\leq i\leq k$. In fact,
$g(\lambda)>0$ holds for $\lambda\geq0$, so $r_{k+1}$ and $r_{k+2}$
cannot be real and positive. Hence they are complex conjugate with
positive real parts. \end{proof} \\

For each $1\leq i\leq k-1$, the signs of $a_{i}$ and $a_{i+1}$
determine the number of jumps of the rational function $g(\lambda)/\prod_{i=1}^{k}(\lambda+b_{i})$
from $\pm\infty$ to $\mp\infty$ as the argument changes from $-b_{i+1}$
to $-b_{i}$. If the signs of $a_{i}$ and $a_{i+1}$ are the same
as in Propositions \ref{prop. 4.1} and \ref{prop. 4.2}, then there is a real root between $-b_{i+1}$
and $-b_{i}$ corresponding to a jump from $\pm\infty$ to $\mp\infty$.
This greatly simplifies the problem of root location which becomes
increasingly complicated if the signs of the $a_{i}$'s are arbitrary,
as can already be seen from the simplest example with $k=2$. In this
case, the characteristic polynomial takes the form \\ 
\begin{align*}
g(\lambda)=p(\lambda,\xi)&=\lambda^{4}+(b_{1}+b_{2})\lambda^{3}+(c^{2}\xi^{2}+b_{1}b_{2})\lambda^{2}
\\
&+\xi^{2}(a_{1}+c^{2}b_{1}+a_{2}+c^{2}b_{2})\lambda+\xi^{2}(c^{2}b_{1}b_{2}+a_{1}b_{2}+a_{2}b_{1})\nonumber
\end{align*}

Let $\Delta_{i}$ denote the $i$th Hurwitz determinant obtained from
the coefficients of the characteristic equation, so that
\begin{eqnarray*}
 &  & \Delta_{1}=b_{1}+b_{2},\ \Delta_{2}=\left|\begin{array}{cc}
b_{1}+b_{2} & g'(0)\\
1 & c^{2}\xi^{2}+b_{1}b_{2}
\end{array}\right|,\ \Delta_{3}=\left|\begin{array}{ccc}
b_{1}+b_{2} & g'(0) & 0\\
1 & c^{2}\xi^{2}+b_{1}b_{2} & g(0)\\
0 & b_{1}+b_{2} & g'(0)
\end{array}\right|,\\
 &  & \Delta_{4}=\left|\begin{array}{cccc}
b_{1}+b_{2} & g'(0) & 0 & 0\\
1 & c^{2}\xi^{2}+b_{1}b_{2} & g(0) & 0\\
0 & b_{1}+b_{2} & g'(0) & 0\\
0 & 1 & c^{2}\xi^{2}+b_{1}b_{2} & g(0)
\end{array}\right|=g(0)\Delta_{3}.
\end{eqnarray*}
\\
\begin{proposition}\label{prop. 4.3} Let $d=1$, $\xi\ne0$ and $a_{1}a_{2}<0$,
then \begin{itemize} \item i) when $g(0)>0$ all eigenvalues have negative real parts if
$\Delta_{2}>0$ and $\Delta_{3}>0$, otherwise two roots have negative
real parts and two roots have non-negative real parts; \item ii) a) when
$g(0)=0$ and $g'(0)>0$, one root is zero and three have negative
real parts if $\Delta_{2}>0$, otherwise there is one zero and one
negative root and two roots with non-negative real parts; b) when
$g(0)=0$ and $g'(0)=0$, two roots with negative real parts and zero
(two-fold); c) when $g(0)=0$ and $g'(0)<0$, there is a zero and
a positive root and two roots with negative real parts; \item iii) when
$g(0)<0$, one root is negative and three roots have positive real
parts if $\Delta_{2}<0$ and $\Delta_{3}<0$, otherwise one root is
positive and three roots have non-positive real parts. \end{itemize}
\end{proposition}

\begin{proof} If $\xi=0$, the eigenvalues are $-b_{1},-b_{2}$
and zero (two-fold). 

i) Let $\xi\ne0$ and assume $g(0)>0$. If $\Delta_{4}\ne0$, according
to the Routh-Hurwitz theorem\cite{gantmacher1959} the number of roots of $g(\lambda)$
in the right half-plane $\Re(\lambda)>0$ is determined by the number
of variations of sign in the sequence 
\[
\left\{ 1,\Delta_{1},\frac{\Delta_{2}}{\Delta_{1}},\frac{\Delta_{3}}{\Delta_{2}},\frac{\Delta_{4}}{\Delta_{3}}\right\} =\left\{ 1,b_{1}+b_{2},\frac{\Delta_{2}}{b_{1}+b_{2}},\frac{\Delta_{3}}{\Delta_{2}},g(0)\right\} 
\]
Hence all the roots of $g(\lambda)$ have negative real parts if and
only if $\Delta_{2}>0$ and $\Delta_{3}>0$. As long as $\Delta_{4}\ne0$,
in all other cases including the singular case $\Delta_{2}=0$ there
are exactly two variations of sign and therefore two roots with positive
real parts, say $r_{1}$ and $r_{2}$. For some real $\alpha,\beta$
we can write 
\[
(\lambda^{2}+\alpha\lambda+\beta)\prod_{i=1}^{2}(\lambda-r_{i})=0.
\]
By Vieta's theorem the roots $r_{3},r_{4}$ of $\lambda^{2}+\alpha\lambda+\beta$
satisfy 
\begin{align*}
& r_{3}+r_{4} = -(b_{1}+b_{2}+r_{1}+r_{2})<0,\\
 & r_{3}\cdot r_{4} = g(0)/r_{1}r_{2}>0
\end{align*}
i.e. $r_{3}$ and $r_{4}$ have negative real parts. Moreover, if $a_{1}<0$
and $a_{2}>0$, these roots are real: one root lies between $-b_{1}$
and $0$ since $g(-b_{1})g(0)=\xi^{2}a_{1}(b_{2}-b_{1})g(0)<0$ and the other
is to the left of $-b_{2}$ because $g(-b_{2})=-\xi^{2}a_{2}(b_{2}-b_{1})<0$
and $\lim_{\lambda\to-\infty}g(\lambda)=+\infty$. 

If $\Delta_{4}=0$, then $\Delta_{3}=0=g'(0)\Delta_{2}-g(0)(b_{1}+b_{2})^{2}$.
Evaluating $\Delta_{2}$ from the latter equation and comparing it
with the original definition of $\Delta_{2}$, we can conclude that
$\Delta_{2}>0$ and $g'(0)>0$. In this case the polynomial enjoys
the following factorization 
\[
\left(\lambda^{2}+\frac{g'(0)}{b_{1}+b_{2}}\right)\left(\lambda^{2}+(b_{1}+b_{2})\lambda+\frac{\Delta_{2}}{b_{1}+b_{2}}\right)=0
\]
 Hence there is a pair of conjugate pure imaginary roots $\pm i\sqrt{g'(0)/(b_{1}+b_{2})}$
and two roots with negative real parts. 

ii) Assume $g(0)=0$ and $\xi\ne0$. Since the constant term is absent,
we can factor out $\lambda$ and reduce the problem of root location
for $g(\lambda)$ to that for which the last Hurwitz determinant is
$\Delta_{3}=g'(0)\Delta_{2}$. a) When $g'(0)>0$ the sequence $\left\{ 1,b_{1}+b_{2},\frac{\Delta_{2}}{b_{1}+b_{2}},g'(0)\right\} $
has no sign variation if $\Delta_{2}>0$ and hence no roots of $\lambda^{3}+(b_{1}+b_{2})\lambda^{2}+(c^{2}\xi^{2}+b_{1}b_{2})\lambda+g'(0)=0$
are in the right-half plane. If $\Delta_{2}<0$, there are two sign
variations and hence two roots with positive real parts and one negative
root. If $\Delta_{2}=0$, the roots are $-(b_{1}+b_{2})$ and $\pm i\sqrt{g'(0)/(b_{1}+b_{2})}$;
b) When $g'(0)=0$, one can factor out $\lambda$ again and obtain
a quadratic equation whose roots have negative real parts; c) When
$g'(0)<0$, then $\Delta_{2}>0$. There is one positive root,
say $r_{1}$, corresponding to a single sign variation in the sequence
$\left\{ 1,b_{1}+b_{2},\frac{\Delta_{2}}{b_{1}+b_{2}},g'(0)\right\} $.
By Vieta's theorem the remaining two roots satisfy 
\begin{align*}
& r_{2}+r_{3} = -(b_{1}+b_{2}+r_{1})<0,\\
& r_{2}\cdot r_{3} = -g'(0)/r_{1}>0
\end{align*}
i.e. $r_{2}$ and $r_{3}$ have negative real parts. 

iii) Assume that $g(0)<0$ and $\xi\ne0$. If $\Delta_{4}\ne0$, $\Delta_{2}<0$
and $\Delta_{3}<0$ by the Routh-Hurwitz theorem there are three roots
in the right-half plane. Since in this case we also have $g'(0)=\frac{\Delta_{3}+g(0)(b_{1}+b_{2})^{2}}{\Delta_{2}}>0$,
it follows from the Descartes' rule of signs that only one of those
three roots is real. In other cases where $\Delta_{4}\ne0$, including
the singular case $\Delta_{2}=0$, there is only one variation of
sign in the sequence $\left\{ 1,\Delta_{1},\frac{\Delta_{2}}{\Delta_{1}},\frac{\Delta_{3}}{\Delta_{2}},\frac{\Delta_{4}}{\Delta_{3}}\right\} $
and hence only one root in the right-half plane. 

If $\Delta_{4}=0$, then $\Delta_{3}=0$ and $g'(0)\Delta_{2}<0$.
From the factorization 
\[
\left(\lambda^{2}+\frac{g'(0)}{b_{1}+b_{2}}\right)\left(\lambda^{2}+(b_{1}+b_{2})\lambda+\frac{\Delta_{2}}{b_{1}+b_{2}}\right)=0
\]
we conclude that there is a pair of conjugate pure imaginary roots
$\pm i\sqrt{g'(0)/(b_{1}+b_{2})}$ and a pair of real roots of opposite
sign if $g'(0)>0$ and $\Delta_{2}<0$. If $g'(0)<0$ and $\Delta_{2}>0$
we have $\pm\sqrt{|g'(0)|/(b_{1}+b_{2})}$ and two roots with negative
real parts. \end{proof} \\

Proposition \ref{prop. 4.3} exhausts all the possibilities for the fourth order
monic polynomial. When more than two $a_{i}$'s have arbitrary signs,
eigenvalues can be studied in a similar manner using higher order
Hurwitz determinants even if some of those determinants vanish. \\

\begin{remark}\label{remark 4.1}
1. Requiring all the $a_{i}$'s to be negative is equivalent to saying
that the relaxation kernel $K(t)=-\sum_{i=1}^{k}a_{i}e^{-b_{i}t}$ is a
totally monotone function.

2. Recall that when deriving the model equation we identified
$a_{i},b_{i}$ and $c^{2}$ with the physical parameters of the system, namely, 
$b_{i}=\tau_{\sigma i}^{-1}>0$,
$c^{2}=\sum_{i=1}^{k}M_{Ri}\tau_{\epsilon i}\tau_{\sigma i}^{-1}>0$
and $a_{i}=M_{Ri}(1-\tau_{\epsilon i}\tau_{\sigma i}^{-1})b_{i}<0$.
Since
\[
g(0)=\xi^{2}\left(c^{2}+\sum_{i=1}^{k}\frac{a_{i}}{b_{i}}\right)\prod_{i=1}^{k}b_{i}=\xi^{2}\prod_{i=1}^{k}b_{i}\sum_{i=1}^{k}M_{Ri}>0,
\]
case iii) of Proposition \ref{prop. 4.1} yielding a positive eigenvalue is
unphysical. In contrast, $g(0)>0$ is fulfilled in Proposition \ref{prop. 4.2}, but
it is assumed that $a_{i}>0$ (no dissipation), so a pair of complex
conjugate roots with positive real parts is also unphysical.

3. The Routh-Hurwitz theorem provides necessary and sufficient
conditions for all of the roots of a polynomial with real coefficients
to lie in the left-half of the complex plane. It allows one to locate the roots just by employing the coefficients of the polynomial which are functions of the parameters controlling the relaxation.

4. Algebraic multiplicities of eigenvalues $\lambda_{j}(\xi)$ remain
constant as $\xi$ ranges along $\mathbb{S}^{d-1}$ and 
$\lambda_{j}(\xi)$ are analytic functions away from the origin,  
admitting a power series expansion in $\xi$. This fact will be used in
Proposition \ref{prop. 4.4} below to investigate the limiting behavior of the eigenvalues as $|\xi|\to0$ and $|\xi|\to\infty$. \\
\end{remark}

\begin{proposition}\label{prop. 4.4} Let $d=1$ and $\lambda_{j}(i\text{\ensuremath{\xi}})$
for $1\leq j\leq k+2$ be the eigenvalues of $\Phi(i\xi)$, then as
$|\xi|\to0$, 
\[
\Re\left(\lambda_{j}(i\xi)\right)=\begin{cases}
-b_{j}-\xi^{2}\frac{a_{j}}{b_{j}^{2}}+O\left(\xi{}^{4}\right) & \mbox{for }j=1,2,\dots,k\\
\pm\Re\left(i\xi\sqrt{c^{2}+\sum_{i=1}^{k}\frac{a_{i}}{b_{i}}}\right)+\xi^{2}\sum_{i=1}^{k}\frac{a_{i}}{2b_{i}^{2}}+O\left(\xi^{3}\right) & \mbox{for }j=k+1,k+2
\end{cases}
\]
if $c^{2}+\sum_{i=1}^{k}\frac{a_{i}}{b_{i}}\ne0$ or 
\[
\Re\left(\lambda_{j}(i\xi)\right)=\begin{cases}
-b_{j}-\xi^{2}\frac{a_{j}}{b_{j}^{2}}+O\left(\xi{}^{4}\right) & \mbox{for }j=1,2,\dots,k\\
\xi^{2}\sum_{i=1}^{k}\frac{a_{i}}{b_{i}^{2}}+O\left(\xi{}^{4}\right) & \mbox{for }j=k+1\\
0+O\left(\xi{}^{4}\right) & \mbox{for }j=k+2
\end{cases}
\]
otherwise, and as $|\xi|\to\infty$ 
\[
\Re\left(\lambda_{j}(i\xi)\right)=\begin{cases}
\Re\left(r_{j}\right)+O\left(\xi{}^{-1}\right) & \mbox{for }j=1,2,\dots,k\\
\frac{1}{2c^{2}}\sum_{i=1}^{k}a_{i}+O\left(\xi{}^{-2}\right) & \mbox{for }j=k+1,k+2
\end{cases}
\]
\end{proposition}

\begin{proof} Let $\zeta=i\xi\in\mathbb{C}$ and recall that $\Phi(\zeta)=-\left(B+\zeta A\right)$.
Following Kato\cite{ghoul}${,}$\cite{kato1976} we treat $-B$ as an unperturbed matrix subjected
to a small perturbation $-\zeta A$. The characteristic equation of
$\Phi(\zeta)$ is an algebraic equation in $\lambda$ of degree $k+2$
and its roots are branches of analytic functions of $\zeta$ with
only algebraic singularities. Hence, in the neighborhood of $\zeta=0$
the following expansion is valid: 
\[
\lambda_{j}(\zeta)=\lambda_{j}^{(0)}+\zeta\lambda_{j}^{(1)}+\zeta^{2}\lambda_{j}^{(2)}+\dots
\]
for $1\leq j\leq k+2$. Here $\lambda_{j}^{(0)}$ are the eigenvalues
of the unperturbed matrix $-B$ and satisfy the equation 
\[
p(\lambda,-i\zeta)\Bigr|_{\zeta=0}=(-1)^{k}\left(\lambda^{(0)}\right)^{2}\prod_{i=1}^{k}\left(\lambda^{(0)}+b_{i}\right)=0
\]
so that 
\[
\lambda_{j}^{(0)}=\begin{cases}
-b_{j} & \mbox{for }j=1,2,\dots,k\\
0 & \mbox{for }j=k+1,k+2
\end{cases}
\]
Solving 
\[
(-1)^{k}\frac{dp}{d\zeta}\biggr|_{\zeta=0}=\lambda^{(0)}\lambda^{(1)}\prod_{i=1}^{k}\left(\lambda^{(0)}+b_{i}\right)\left(2+\sum_{i=1}^{k}\frac{\lambda^{(0)}}{\lambda^{(0)}+b_{i}}\right)=0\ \ \ \mbox{and}\ \ \ \frac{d^{2}p}{d\zeta^{2}}\biggr|_{\zeta=0}=0
\]
one obtains 
\[
\lambda_{j}^{(1)}=\begin{cases}
0 & \mbox{for }j=1,2,\dots,k\\
\pm\sqrt{c^{2}+\sum_{i=1}^{k}\frac{a_{i}}{b_{i}}} & \mbox{for }j=k+1,k+2
\end{cases}
\]
The next order correction comes from solving $\frac{d^{2}p}{d\zeta^{2}}\Bigr|_{\zeta=0}=0$
and $\frac{d^{3}p}{d\zeta^{3}}\Bigr|_{\zeta=0}=0$, thus 
\[
\lambda_{j}^{(2)}=\begin{cases}
\frac{a_{j}}{b_{j}^{2}} & \mbox{for }j=1,2,\dots,k\\
-\sum_{i=1}^{k}\frac{a_{i}}{b_{i}^{2}} & \mbox{for }j=k+1\\
0 & \mbox{for }j=k+2
\end{cases}\ \ \ \mbox{or}\ \ \ \lambda_{j}^{(2)}=\begin{cases}
\frac{a_{j}}{b_{j}^{2}} & \mbox{for }j=1,2,\dots,k\\
-\sum_{i=1}^{k}\frac{a_{i}}{2b_{i}^{2}} & \mbox{for }j=k+1,k+2
\end{cases}
\]
depending on whether $c^{2}+\sum_{i=1}^{k}\frac{a_{i}}{b_{i}}=0$
or not, respectively. Equation $\frac{d^{3}p}{d\zeta^{3}}\biggr|_{\zeta=0}=0$
also implies that $\lambda_{j}^{(3)}=0$ for $1\leq j\leq k$. 

When $|\xi|\to\infty$ we can write $\Phi(\zeta)=-\left(B+\zeta A\right)=-\zeta\left(A+\zeta^{-1}B\right)$
and consider $-\zeta^{-1}B$ to be a small perturbation of $-A$.
The eigenvalues $\mu_{j}\left(\zeta^{-1}\right)$ of $A+\zeta^{-1}B$
are related to those of $\Phi(\zeta)$ by $\lambda_{j}(\zeta)=\zeta\mu_{j}\left(\zeta^{-1}\right)$.
The characteristic polynomial of $A+\nu B$ is 
\[
q(\mu,\nu)=(-1)^{k}\left(\mu^{2}-c^{2}-\sum_{i=1}^{k}\frac{\nu a_{i}}{\mu+\nu b_{i}}\right)\prod_{i=1}^{k}(\mu+\nu b_{i})
\]
where $\nu=\zeta^{-1}$. In the neighborhood of $\nu=0$ we have 
\[
\mu_{j}(\nu)=\mu_{j}^{(0)}+\nu\mu_{j}^{(1)}+\nu^{2}\mu_{j}^{(2)}+\dots
\]
for $1\leq j\leq k+2$. The eigenvalues of $-A$ satisfy $q\left(\mu,\nu\right)\Bigr|_{\nu=0}=0$,
hence 
\[
\mu_{j}^{(0)}=\begin{cases}
0 & \mbox{for }j=1,2,\dots,k\\
\pm c & \mbox{for }j=k+1,k+2
\end{cases}
\]
Computing $\frac{dq}{d\nu}\biggr|_{\nu=0}=0$ we find 
\[
(-1)^{k}\left(\mu^{(0)}\right)^{k-1}\left(\left(\left(\mu^{(0)}\right)^{2}-c^{2}\right)\sum_{i=1}^{k}\left(\mu^{(1)}+b_{i}\right)+2\left(\mu^{(0)}\right)^{2}\mu^{(1)}-\sum_{i=1}^{k}a_{i}\right)=0
\]
so that 
\[
\mu_{j}^{(1)}=\begin{cases}
r_{j} & \mbox{for }j=1,2,\dots,k\\
\frac{1}{2c^{2}}\sum_{i=1}^{k}a_{i} & \mbox{for }j=k+1,k+2
\end{cases}
\]
where $r_{j}$ are roots of
$\frac{d^{k}q}{d\nu^{k}}\biggr|_{\nu=0,\mu_{j}^{(0)}=0}=0$.  One can
show
that $$\mu_{j}^{(2)}=\mp\left(\frac{1}{2c^{3}}\sum_{i=1}^{k}a_{i}b_{i}+\frac{3}{8c^{5}}\left(\sum_{i=1}^{k}a_{i}\right)^{2}\right)$$
for $j=k+1,k+2$ by solving $\frac{d^{2}q}{d\nu^{2}}\biggr|_{\nu=0}=0$.
\end{proof}

In higher dimensions the analysis is similar but lengthier, remember that $\xi$ should be
replaced by $|\xi|$ therein. 

\section{Stability}\label{Sec.5}

Well-posedness of the Cauchy problem described in Definition \ref{def. 3.2} does
not rule out the possibility of exponential growth of solutions as
time approaches infinity unless $\alpha$ is arbitrarily small or negative.
The following definition helps to eliminate exponential instabilities. \\

\begin{definition}\label{def. 5.1}
The Cauchy problem for a constant coefficient
operator $\mathcal{L}$ is weakly (strongly) stable if it is weakly
or strongly well-posed and the solution $U(t)$ satisfies 
\[
\left\Vert U(t)\right\Vert _{L^{2}\left(\mathbb{R}^{d}\right)}\leq C(1+t)^{s}\left\Vert U_{0}\right\Vert _{H^{s}\left(\mathbb{R}^{d}\right)},\ \ \ t\geq0
\]
with $C>0$ and $s>0$ ($s=0$). 
\end{definition} \\

A necessary and sufficient condition for weak stability is that all
eigenvalues $\lambda_{j}(\xi)$ of $\Phi(i\xi)=-\left(B+iA(\xi)\right)$
satisfy $\Re\left(\lambda_{j}(\xi)\right)\leq0$. Furthermore, if the
Jordan blocks corresponding to the eigenvalues with
$\Re\left(\lambda_{j}(\xi)\right)=0$ are trivial, then the problem is
strongly stable  (cf. Lemma 2.1 in \cite{solem}). \\

\begin{theorem}\label{th. 5.1}
Let $d\geq1$, $\xi\ne0$, $a_{i}<0$ for all $1\leq i\leq k$ and $g(0)\geq0$,
then the Cauchy problem for a constant coefficient operator
$\mathcal{L}$ is strongly stable.
\end{theorem}

\begin{proof} By Proposition \ref{prop. 4.1}, all eigenvalues of $\Phi(i\xi)$
satisfy $\Re\left(\lambda_{j}(\xi)\right)\leq0$. Since the characteristic
polynomial in higher dimensions splits as in Eq. \eqref{eq: 4.1}, $\tilde{p}(\lambda,\xi_{1},\dots,\xi_{d})=p(\lambda,|\xi|)\lambda^{d-1}\prod_{i=1}^{k}(\lambda+b_{i})^{d-1}$,
the algebraic multiplicity $m$ of the zero eigenvalue is $m=d-1$
when $g(0)>0$ and $m=d$ if $g(0)=0$. In both cases, $m$ does not
change as $\xi$ ranges along $\mathbb{S}^{d-1}$ and moreover algebraic multiplicity
is equal to the geometric multiplicity. Note that $\lambda(\xi)=0$
solves $\frac{\partial^{m-1}}{\partial\lambda^{m-1}}\tilde{p}(\lambda,\xi_{1},\dots,\xi_{d})=0$,
but $\xi\cdot\nabla_{\xi}\frac{\partial^{m-1}}{\partial\lambda^{m-1}}\tilde{p}(\lambda,\xi_{1},\dots,\xi_{d})\ne0$
at $\lambda=0$. Hence $\Phi(i\xi)$ is of principal type at $\lambda=0$
and the Jordan blocks corresponding to zero eigenvalues are all trivial.
By Parseval's relation 
\begin{align}
\left\Vert U(t)\right\Vert _{L^{2}\left(\mathbb{R}^{d}\right)}&=\left\Vert e^{t\Phi(i\xi)}\hat{U}_{0}(\xi)\right\Vert _{L^{2}\left(\mathbb{R}^{d}\right)} =\left\Vert P^{-1}(\xi)e^{tJ}P(\xi)\hat{U}_{0}(\xi)\right\Vert _{L^{2}\left(\mathbb{R}^{d}\right)} \nonumber
\\
&\leq\left\Vert P^{-1}(\xi)\right\Vert \left\Vert P(\xi)\right\Vert \left\Vert U_{0}\right\Vert _{L^{2}\left(\mathbb{R}^{d}\right)}, \nonumber
\end{align}
where $J$ is the Jordan matrix. By Theorem \ref{th. 3.1}  the Cauchy problem
for $\mathcal{L}$ is strongly well-posed and Proposition \ref{prop. 4.4} implies
that $\Re\left(\lambda_{j}(\xi)\right)\nrightarrow+\infty$ as $|\xi|\to+\infty$,
so $\left\Vert P^{-1}(\xi)\right\Vert \left\Vert P(\xi)\right\Vert $
is bounded by a constant $C>0$ and the claim follows. \end{proof} \\

\begin{remark}\label{remark 5.1}
For $\Phi(i\xi)$ to be of principal type at $\lambda=0$, it is
important that the constant algebraic multiplicity is equal to the
geometric multiplicity. Consider, for example, case ii) (b) of
Proposition \ref{prop. 4.3}: the geometric multiplicity of the zero eigenvalue is
less than its algebraic multiplicity in any dimension $d\geq1$.  Hence
$\Phi(i\xi)$ is not of principal type at $\lambda=0$ and the Jordan
matrix contains a non-trivial block. The Cauchy problem for
$\mathcal{L}$ is only weakly stable in that case. 

 M\'{e}tivier and Zumbrun \cite{metivierzum} classify the multiple eigenvalues as algebraically regular, geometrically regular and nonregular. Eigenvalues of constant multiplicity are algebraically regular. If in addition they are semi-simple, then they are geometrically regular. Simple roots are geometrically regular by definition. Geometric regularity implies Majda's block structure condition and provides an optimal characterization of this condition.

\end{remark}

\subsection*{Energy decay}
Consider for a moment the following viscoelastic wave equation 
\begin{equation}\label{eq: 5.1}
\phi_{tt}-c^{2}\Delta\phi+\int_{0}^{t}K(t-s)\Delta\phi(s)ds=0, \ \ \  x\in\mathbb{R}^d, 
\end{equation}
together with the associated standard energy in Fourier space  
\[
\hat{E}(\xi,t)=\frac{1}{2}|\hat{\phi}_{t}|^{2}+\frac{1}{2}c^{2}|\xi|^{2}|\hat{\phi}|^{2}
\]
The following assumptions on the relaxation kernel $K(t)$ are commonly accepted in the literature: \\
\begin{enumerate}
\item[] {\bf [A2]} $K$: $\mathbb{R}^{+}\to\mathbb{R}^{+}$ is a non-increasing $\mathcal{C}^{1}$
function and $l=c^{2}-\int_{0}^{\infty}K(s)ds>0,$ 
\item[] {\bf[A3]} $K(0)>0$ and $K'(t)<0$ for all $t\geq0$. \\
\end{enumerate} 

Examples of kernels satisfying the above assumptions are $K(0){(1+t)}^{-\nu}$, $K(0)e^{-(1+t)^\nu}$ with properly chosen $\nu>1$ and $K(0)>0$. 

Assumption {\bf [A2]} has a physical origin: in statics, i.e. when $\sigma(x,t)=\bar{\sigma}(x)$ and $\epsilon(x,t)=\bar{\epsilon}(x)$ Eq. \eqref{eq: 2.3} reduces to \[\bar{\sigma}(x)=\rho\left(c^{2}-\int_{0}^{\infty}K(s)ds\right)\bar{\epsilon}(x)\] so {\bf [A2]} states that in a viscoelastic medium the equilibrium stress modulus is positive (cf. Eq. (75) in Ref. \cite{renardy} where the equilibrium stress function is considered). \\

\begin{theorem}\label{th. 5.2}
Assume that $K(t)$ satisfies {\bf [A2]} and {\bf [A3]}, then the energy of the solution to \eqref{eq: 5.1} decreases in time.  
\end{theorem}
\begin{proof}
Multiplying the Fourier transform of \eqref{eq: 5.1},
\[
\hat{\phi}_{tt}+c^{2}|\xi|^{2}\hat{\phi}-\int_{0}^{t}|\xi|^{2}K(t-s)\hat{\phi}(s)ds=0,
\]
by $\hat{\phi}_{t}^{*}$ and taking the real part we compute 
\[
\frac{1}{2}\frac{d}{dt}\left(|\hat{\phi}_{t}|^{2}+c^{2}|\xi|^{2}|\hat{\phi}|^{2}\right)=|\xi|^{2}\Re\left(\hat{\phi}_{t}^{*}\int_{0}^{t}K(t-s)\hat{\phi}(s)ds\right)=|\xi|^{2}\Re\left(\hat{\phi}_{t}^{*}(K*\hat{\phi})(t)\right)
\]
where we utilized the first of the following convolutions: 
\begin{align*}
 & (K*f)(t)=\int_{0}^{t}K(t-s)f(s)ds,\\
 & (K\circledast f)(t)=\int_{0}^{t}K(t-s)|f(s)-f(t)|^{2}ds,
\end{align*}
defined for any real or complex valued function $f(t)$. Using the
second definition one can compute 
\[
\frac{d}{dt}\left((K\circledast\hat{\phi})(t)-|\hat{\phi}|^{2}\int_{0}^{t}K(s)ds\right)=(K'\circledast\hat{\phi})(t)-2\Re\left(\hat{\phi}_{t}^{*}(K*\hat{\phi})(t)\right)-K(t)|\hat{\phi}|^{2}
\]
Hence substituting the previously obtained expression for $\Re\left(\hat{\phi}_{t}^{*}(K*\hat{\phi})(t)\right)$ and using {\bf [A2]} and {\bf [A3]} we have
\begin{align}\label{eq: 5.2}
\frac{1}{2}\frac{d}{dt}\left(|\hat{\phi}_{t}|^{2}+|\xi|^{2}|\hat{\phi}|^{2}\left(c^{2}-\int_{0}^{t}K(s)ds\right)+|\xi|^{2}(K\circledast\hat{\phi})(t)\right)= \nonumber \\   |\xi|^2\left((K'\circledast\hat{\phi})(t)-K(t)|\hat{\phi}|^{2}\right)\leq 0
\end{align}
By introducing the following functional: 
\[
0\leq\hat{\mathcal{E}}(\xi,t)=\frac{1}{2}|\hat{\phi}_{t}|^{2}+\frac{1}{2}\left(c^{2}-\int_{0}^{t}K(s)ds\right)|\xi|^{2}|\hat{\phi}|^{2}+\frac{1}{2}|\xi|^{2}(K\circledast\hat{\phi})(t)
\]
then, by \eqref{eq: 5.2}, $\hat{\mathcal{E}}(\xi,t)$ is non-increasing and obeys $\hat{\mathcal{E}}(\xi,t)\leq\hat{\mathcal{E}}(\xi,0)=\hat{E}(\xi, 0)$, and
on the other hand $\hat{E}(\xi,t)\leq c^{2}l^{-1}\hat{\mathcal{E}}(\xi,t)$
so the uniform decay of $\hat{\mathcal{E}}$ implies the uniform decay
of $\hat{E}$. 
\end{proof} \\

In the present manuscript we have dealt with the kernel $K(t)=-\sum_{i=1}^{k}a_{i}e^{-b_{i}t}$ with constant coefficients and  in this case \eqref{eq: 2.5} reduces to \eqref{eq: 5.1}.  If we choose $a_i<0$ and \[c^{2}-\int_{0}^{\infty}K(s)ds=c^2+\sum_{i=1}^k{a_i\over b_i}=g(0)>0,\] then assumptions {\bf [A2]}, {\bf [A3]} hold true and Theorem \ref{th. 5.2} shows that the energy of an absorbing medium dissipates over time. Moreover, since $-K'(t)/K(t)$ is bounded from below by a positive constant: $-K'(t)/K(t)> b_{k}=\max_{i}b_{i}>0$, it is possible to show that the energy decays exponentially\cite{pata}${,}$\cite{saidhouari}.

Finally note that the absorption condition, ${\bf D-C}>0$, stated in Ref. \cite{becache2005mixed} is equivalent to $a_i<0$. Indeed, by definition $a_i<0$ holds whenever $\tau_{\epsilon i}>\tau_{\sigma i}$ (see Remark \ref{remark 4.1}). In the mono-dimensional Zener model ${\bf C}=\mu$ and ${\bf D}=\mu\tau_{\epsilon}/\tau_{\sigma}$, therefore, ${\bf D-C}>0$ reduces $\tau_{\epsilon}>\tau_{\sigma}$. In higher dimensions, $\tau_{\epsilon i}>\tau_{\sigma i}$ for all $1\leq i\leq k$, ensures that the tensor ${\bf D-C}$ is positive definite. 
\\ \\ \\ \\ \\ 
          % Non-BibTeX users please use

          \end{document}